\documentclass[11pt]{article}
\usepackage[utf8]{inputenc}
\usepackage{amsthm}
\usepackage{amsfonts}
\usepackage{amsmath}
\usepackage{amssymb}
\usepackage{colonequals}
\usepackage{color}
\usepackage{epsfig}
\usepackage{url}
\usepackage{hyperref}
\newtheorem{theorem}{Theorem}
\newtheorem{problem}[theorem]{Problem}
\newtheorem{corollary}[theorem]{Corollary}
\newtheorem{lemma}[theorem]{Lemma}

\newcommand{\eps}{\varepsilon}
\newcommand{\ch}{\text{ch}}

\newcommand{\brm}[1]{\operatorname{#1}}

\newenvironment{subproof}{%
  \begin{proof}[Subproof]%
}{%
  \end{proof}%
}

\usepackage{xcolor}
\usepackage[textsize=footnotesize]{todonotes}
\usepackage{xspace}							

\title{Flexibility of triangle-free planar graphs\thanks{Work on this paper was supported by project 17-04611S (Ramsey-like aspects of graph coloring) of Czech Science Foundation.}}
\author{Zden\v{e}k Dvo\v{r}\'ak\thanks{Computer Science Institute (CSI), Charles University in Prague.
E-mail: \protect\href{mailto:rakdver@iuuk.mff.cuni.cz}{\protect\nolinkurl{rakdver@iuuk.mff.cuni.cz}}.}\and
Tom\'a\v{s} Masa\v{r}\'ik\thanks{Department of Applied Mathematics, Charles University in Prague.
E-mail: \protect\href{mailto:masarik@kam.mff.cuni.cz}{\protect\nolinkurl{masarik@kam.mff.cuni.cz}}. Partially supported by the project SVV–2017–260452.}\and
Jan Mus\'ilek\thanks{Department of Applied Mathematics, Charles University in Prague.
E-mail: \protect\href{mailto:stinovlas@kam.mff.cuni.cz}{\protect\nolinkurl{stinovlas@kam.mff.cuni.cz}}.}\and
Ond\v{r}ej Pangr\'ac\thanks{Computer Science Institute (CSI), Charles University in Prague.
E-mail: \protect\href{mailto:pangrac@iuuk.mff.cuni.cz}{\protect\nolinkurl{pangrac@iuuk.mff.cuni.cz}}.}}

\date{}

\begin{document}
\maketitle

\begin{abstract}
Let $G$ be a planar graph with a list assignment $L$.  Suppose a preferred color is given for some of the
vertices.  We prove that if $G$ is triangle-free and all lists have size at least four, then there exists
an $L$-coloring respecting at least a constant fraction of the preferences.
\end{abstract}

\section{Introduction}

In a proper graph coloring, we want to assign to each vertex of a graph one of a fixed number of colors
in such a way that adjacent vertices receive distinct colors.  Proper graph coloring models a number of
real-world problems related to scheduling or distributing limited resources in a way that avoids conflicts
(e.g., scheduling classes into time slots so that any two classes taught by the same teacher occur at different
times, or the compiler assigning variables to registers so that any two variables which are used at the same time
reside in different registers).

In such applications, it is common for the vertices to prefer to be colored by certain colors (e.g., teachers
may prefer to teach or not to teach at certain times, the variables may be more profitably kept in specific registers
in case some assembler instructions can only be applied to those registers).  Usually, it is not possible
to satisfy all such preferences.  This motivates the following definitions (which we present in a more general list
coloring setting, for a reason we discuss below).

A \emph{list assignment} $L$ for a graph $G$ is a function that to each vertex $v\in V(G)$ assigns
a set $L(v)$ of colors, and an \emph{$L$-coloring} is a proper coloring $\varphi$ such that $\varphi(v)\in L(v)$ for all $v\in V(G)$.
A graph $G$ is \emph{$k$-choosable} if $G$ is $L$-colorable from every assignment $L$ of lists of size at least $k$.
A \emph{weighted request} is a function $w$ that to each pair $(v,c)$ with $v\in V(G)$ and $c\in L(v)$
assigns a nonnegative real number.  Let $w(G,L)=\sum_{v\in V(G),c\in L(v)} w(v,c)$.
For $\eps>0$, we say that $w$ is \emph{$\eps$-satisfiable} if there exists an $L$-coloring $\varphi$ of $G$ such that
$$\sum_{v\in V(G)} w(v,\varphi(v))\ge\eps\cdot w(G,L).$$
An important special case is when at most one color can be requested at each vertex and all such colors have the
same weight (say $w(v,c)=1$ for at most one color $c\in L(v)$, and $w(v,c')=0$ for any other color $c'$):
A \emph{request} for a graph $G$ with a list assignment $L$ is a function $r$ with $\brm{dom}(r)\subseteq V(G)$
such that $r(v)\in L(v)$ for all $v\in\brm{dom}(r)$.  For $\eps>0$, a request $r$ is \emph{$\eps$-satisfiable}
if there exists an $L$-coloring $\varphi$ of $G$ such that $\varphi(v)=r(v)$ for at least $\eps|\brm{dom}(r)|$ vertices $v\in\brm{dom}(r)$.

In particular, a request $r$ is $1$-satisfiable if and only if the precoloring given by $r$ extends to an $L$-coloring of $G$.
The corresponding \emph{precoloring extension} problem has been studied in a number of contexts: as a tool to deal with
small cuts in the considered graph by coloring one part of the graph recursively and then extending the corresponding
precoloring of the cut vertices to the other part~\cite{arnborg1989linear,trfree7,Th2}, as a way to show that
a graph has many different colorings\cite{cylgen-part2,lukethe}, or from the algorithmic complexity perspective~\cite{latin,fiala2003np}.
In planar graphs, it is known that precoloring of any set of vertices at distance at least three from one another extends when
at least 5 colors are used~\cite{Alb98} and for sufficiently distant vertices this holds also in the list coloring setting~\cite{LidDvoMohPos};
on the other hand, a precoloring of two arbitrarily distant vertices of a planar graph does not necessarily extend to a $4$-coloring~\cite{Fisk78}.

Dvořák, Norin and Postle~\cite{requests} asked a related question: In a given class of graphs, is it always possible to satisfy at least a constant proportion
of the requests?  We say that a graph $G$ with the list assignment $L$ is \emph{$\eps$-flexible} if every request is $\eps$-satisfiable,
and it is \emph{weighted $\eps$-flexible} if every weighted request is $\eps$-satisfiable (of course, weighted $\eps$-flexibility implies $\eps$-flexibility).
Dvořák, Norin and Postle~\cite{requests} established some basic properties of the concept and proved that several interesting graph classes are flexible:
\begin{itemize}
\item For every $d\ge 0$, there exists $\varepsilon>0$ such that $d$-degenerate graphs with assignments of lists of size $d+2$ are weighted $\eps$-flexible.
\item There exists $\varepsilon>0$ such that every planar graph with assignment of lists of size $6$ is $\eps$-flexible.
\item There exists $\varepsilon>0$ such that every planar graph of girth at least five with assignment of lists of size $4$ is $\eps$-flexible.
\end{itemize}
They also raised a number of interesting questions, including the following one.
\begin{problem}\label{prob-dnp}
Does there exists $\varepsilon>0$ such that every planar graph $G$ and assignment $L$ of lists of size
\begin{itemize}
\item[(a)] five in general,
\item[(b)] four if $G$ is triangle-free,
\item[(c)] three if $G$ has girth at least five
\end{itemize}
is (weighted) $\eps$-flexible?
\end{problem}
Let us remark that planar graphs are $5$-choosable~\cite{thomassen1994} but not necessarily $4$-choosable~\cite{voigt1993},
triangle-free planar graphs are $4$-choosable by a simple degeneracy argument but not necessarily $3$-choosable~\cite{voigt1995},
and planar graphs of girth at least $5$ are $3$-choosable~\cite{ThoShortlist}.  Also, let us remark that the analogous questions
in the ordinary proper coloring setting are trivial: If all vertices of a $k$-colorable graph $G$ are assigned the same list of length $k$, then $G$
with this uniform list assignment is weighted $k^{-1}$-flexible, as is easy to see by considering the colorings of $G$
arising from a fixed $k$-coloring by permuting the colors~\cite{requests}.

We answer the part (b) of Problem~\ref{prob-dnp} in positive.
\begin{theorem}\label{thm-main}
There exists $\eps>0$ such that each planar triangle-free graph with assignment of lists of size four is weighted $\eps$-flexible.
\end{theorem}

Let us remark that the underlying choosability result for Theorem~\ref{thm-main} is a trivial
average degree argument.  While the proof of the flexibility result also exploits bounded average degree of the triangle-free planar graphs,
it somewhat unexpectedly turns out to require much more involved reducibility and discharging arguments.

\section{Flexibility and reducible configurations}

To prove weighted $\eps$-flexibility, we use the following observation made by Dvo\v{r}\'ak et al.~\cite{requests}.
\begin{lemma}\label{lemma:distrib}
Let $G$ be a graph and let $L$ be a list assignment for $G$.
Suppose $G$ is $L$-colorable and there exists a probability distribution on $L$-colorings $\varphi$ of $G$
such that for every $v\in V(G)$ and $c\in L(v)$, $\brm{Prob}[\varphi(v)=c]\ge\eps$.
Then $G$ with $L$ is weighted $\eps$-flexible.
\end{lemma}

Let $H$ be a graph. For a positive integer $d$, a set $I\subseteq V(H)$ is \emph{$d$-independent} if the distance between any distinct vertices
of $I$ in $H$ is greater than $d$.  Let $1_I$ denote the characteristic function of $I$, i.e., $1_I(v)=1$ if $v\in I$ and $1_I(v)=0$ otherwise.
For functions that assign integers to vertices of $H$, we define addition and subtraction in the natural way,
adding/subtracting their values at each vertex independently.
For a function $f:V(H)\to\mathbb{Z}$ and a vertex $v\in V(H)$, let $f\downarrow	v$ denote the function such that $(f\downarrow v)(w)=f(w)$ for $w\neq v$
and $(f\downarrow v)(v)=1$.  A list assignment $L$ is an \emph{$f$-assignment} if $|L(v)|\ge f(v)$ for all $v\in V(H)$.

Suppose $H$ is an induced subgraph of another graph $G$.  For an integer $k\ge 3$, let $\delta_{G,k}:V(H)\to\mathbb{Z}$
be defined by $\delta_{G,k}(v)=k-\deg_G(v)$ for each $v\in V(H)$.  For another integer $d\ge 0$, we say that $H$ is a \emph{$(d,k)$-reducible}
induced subgraph of $G$ if
\begin{itemize}
\item[(FIX)] for every $v\in V(H)$, $H$ is $L$-colorable for every $((\deg_H+\delta_{G,k})\downarrow v)$-assignment $L$, and
\item[(FORB)] for every $d$-independent set $I$ in $H$ of size at most $k-2$, $H$ is $L$-colorable for every $(\deg_H+\delta_{G,k}-1_I)$-assignment $L$.
\end{itemize}
Note that (FORB) in particular implies that $\deg_H(v)+\delta_{G,k}(v)\ge 2$ for all $v\in V(H)$.
Before we proceed, let us give an intuition behind these definitions.  Consider any assignment $L_0$ of lists of size $k$
to vertices of $G$. The function $\delta_{G,k}$ describes how many more (or fewer) available colors each vertex has compared to its degree.
Suppose we $L_0$-color $G-V(H)$, and let $L'$ be the list assignment for $H$ obtained from $L_0$ by removing from the list of each vertex
the colors of its neighbors in $V(G)\setminus V(H)$.  In $L'$, each vertex $v\in V(H)$ has at least $\deg_H(v)+\delta_{G,k}(v)$
available colors, since each color in $L_0(v)\setminus L'(v)$ corresponds to a neighbor of $v$ in $V(G)\setminus V(H)$.
Hence, (FIX) requires that $H$ is $L'$-colorable even if we prescribe the color of any single vertex of $H$,
and (FORB) requires that $H$ is $L'$-colorable even if we forbid to use one of the colors on the $d$-independent set $I$.

The following lemma is implicit in Dvo\v{r}\'ak et al.~\cite{requests}; we include a proof for completeness.
\begin{lemma}\label{lemma-redu}
For all integers $g,k\ge 3$ and $b\ge 1$, there exists $\eps>0$ as follows.  Let $G$ be a graph of girth at least $g$. 
If for every $Z\subseteq V(G)$,
the graph $G[Z]$ contains an induced $(g-3,k)$-reducible subgraph with at most $b$ vertices, then $G$ with any assignment of lists
of size $k$ is weighted $\eps$-flexible.
\end{lemma}
\begin{proof}
Let $p=k^{-b}$ and $\eps=p^{k-1}$.
For a graph $G$ satisfying the assumptions and an assignment $L$ of lists of size $k$,
we prove the following claim by induction on the number of vertices; the part (i) implies that
$G$ with $L$ is weighted $\eps$-flexible by Lemma~\ref{lemma:distrib}:

\smallskip

There exists a probability distribution on $L$-colorings $\varphi$ of $G$ such that
\begin{itemize}
\item[(i)] for every $v\in V(G)$ and a color $c\in L(v)$, the probability that $\varphi(v)=c$ is at least $\eps$, and
\item[(ii)] for every color $c$ and every $(g-3)$-independent set $I$ in $G$ of size at most $k-2$, the probability
that $\varphi(v)\neq c$ for all $v\in I$ is at least $p^{|I|}$.
\end{itemize}
The claim clearly holds for a graph with no vertices, the basic case of the induction.
Hence, suppose that $V(G)\neq\emptyset$.  By the assumptions, there exists $Y\subseteq V(G)$ of size at most $b$
such that $G[Y]$ is $(g-3,k)$-reducible.  By the induction hypothesis, there exists a probability distribution
on $L$-colorings of $G-Y$ satisfying (i) and (ii).  Choose an $L$-coloring $\psi$ from this distribution and
let $L'$ be the list assignment for $G[Y]$ defined
by $L'(y)=L(y)\setminus \{\psi(v):v\in V(G-Y), vy\in E(G)\}$.  Note that $|L'(y)|\ge \deg_{G[Y]}(y)+\delta_{G,k}(y)$
for all $y\in Y$, and thus $G[Y]$ has an $L'$-coloring by (FORB) applied with $I=\emptyset$. Among all $L'$-colorings of $G[Y]$, choose one
uniformly at random, extending $\psi$ to an $L$-coloring $\varphi$ of $G$.

Let us first argue that (ii) holds.  Let $I_1=I\setminus Y$ and $I_2=I\cap Y$.  By the induction hypothesis,
we have $\varphi(v)\neq c$ for all $v\in I_1$ with probability at least $p^{|I_1|}$.  If $I_2=\emptyset$, this implies (ii).
Hence, suppose that $|I_2|\ge 1$.  For $y\in I_2$, let $L_c(y)=L'(y)\setminus\{c\}$,
and for $y\in Y\setminus I_2$, let $L_c(y)=L'(y)$.  Note that $|L_c(y)|\ge \deg_{G[Y]}(y)+\delta_{G,k}(y)-1_I(y)$ for all $y\in Y$,
and by (FORB), $G[Y]$ has an $L_c$-coloring.  Since $G[Y]$ has at most $k^b$ $L'$-colorings, we conclude that
the probability that $\varphi(y)\neq c$ for all $y\in I_2$ is at least $1/k^b=p\ge p^{|I_2|}$.
Hence, the probability that $\varphi(y)\neq c$ for all $y\in I$ is at least $p^{|I_1|+|I_2|}\ge p^{|I|}$, implying (ii).

Next, let us argue that (i) holds.  For $v\in V(G)\setminus Y$, this is true by the induction hypothesis.
Hence, suppose that $v\in Y$, and let $I$ be the set of neighbors of $v$ in $V(G)\setminus Y$.
Since $G$ has girth at least $g$ and all vertices in $I$ have a common neighbor, the set $I$ is $(g-3)$-independent in $G-Y$.
Furthermore, (FORB) implies $1\le \deg_{G[Y]}(v)+\delta_{G,k}(v)-1_{\{v\}}(v)=\deg_{G[Y]}(v)+k-\deg_G(v)-1=k-1-|I|$,
and thus $|I|\le k-2$.  Hence, by the induction hypothesis we have $\psi(u)\neq c$ for all $u\in I$ with probability at
least $p^{k-2}$.  Assuming this is the case, (FIX) implies there exists an $L'$-coloring of $G[Y]$ which gives $v$ the color $c$.
Since $G[Y]$ has at most $k^b$ $L'$-colorings, we conclude that the probability that $\varphi(v)=c$ is at least $p^{k-2}/k^b=\eps$.
Hence, (i) holds.
\end{proof}

The argument used to prove Lemma~\ref{lemma-redu} also implies the following fact.
\begin{lemma}\label{lemma-many}
Let $k\ge 3$ and $b\ge 1$ be integers and let $G$ be a graph.
If for every $Z\subseteq V(G)$,
the graph $G[Z]$ contains an induced $(b,k)$-reducible subgraph with at most $b$ vertices, then $G$ has at least $2^{|V(G)|/b}$ colorings
from any assignment of lists of size $k$.
\end{lemma}
Note that in subgraphs with at most $b$ vertices, $(b,k)$-reducibility means that in the (FORB) property, we only care about sets $I$ of size $1$.
In particular, the arguments from Sections~\ref{sec-four} 
imply that planar triangle-free graphs have exponentially many colorings from lists of size four.
Of course, there exist simpler proofs of this fact, see e.g.~\cite{exp4trfree} for the triangle-free case
(even in a more general setting of graphs on surfaces).

\section{Triangle-free planar graphs}\label{sec-four}

In this section, we prove Theorem~\ref{thm-main}.  The proof is by the discharging method: We first describe a number of configurations
ensuring the existence of a small $(1,4)$-reducible subgraph, then perform a double-counting argument to show that one of these configurations
appears in any triangle-free planar graph, so that Lemma~\ref{lemma-redu} applies.

\subsection{List coloring preliminaries}

We use the following well-known fact.
\begin{lemma}[Thomassen~\cite{Thomassen97}]\label{lemma-gallai}
Let $G$ be a connected graph and $L$ a list assignment such that $|L(u)|\ge \deg(u)$ for all
$u\in V(G)$.  If either there exists a vertex $u\in V(G)$ such that $|L(u)|>\deg(u)$, or
some $2$-connected component of $G$ is neither complete nor an odd cycle, then $G$ is $L$-colorable.
\end{lemma}

This has the following consequence.
\begin{corollary}\label{cor-rem}
Let $G$ be a connected graph and let $v$ be a vertex of $G$.
Let $L$ be an assignment of non-empty lists to vertices of $G$ such that $G-v$
is $L$-colorable.  Let $C_1$, \ldots, $C_k$ be the vertex sets of the components
of $G-v$, and for $1\le i\le k$, let $n_i$ be the number of neighbors of $v$ in $C_i$.
Let $c_i=n_i-1$ if $|L(x)|\ge\deg_G(x)$ for all $x\in C_i$ and $G[C_i\cup\{v\}]$ has a $2$-connected component
that is neither complete nor an odd cycle, and $c_i=n_i$ otherwise.  If $|L(v)|>c_1+\ldots+c_k$, then $G$ is $L$-colorable.
\end{corollary}
\begin{proof}
For $i=1,\ldots,k$, let $G_i=G[C_i\cup\{v\}]$ and let $F_i$ consist of the colors $c\in L(v)$ such that $G_i$ does not have an $L$-coloring that assigns the color $c$ to $v$.
We claim that $|F_i|\le c_i$; if we prove this to be the case, the claim follows, since then there exists a color $c'\in L(v)\setminus (F_1\cup\ldots\cup F_k)$,
by the definition of the sets $F_i$ there exists an $L$-coloring of $G_i$ assigning to $v$ the color $c'$ for $i=1, \ldots, k$, and
the combination of these colorings gives an $L$-coloring of $G$.

Let $L_i$ be the list assignment for $G_i$ such that $L_i(x)=L(x)$ for $u\in C_i$
and $L_i(v)=F_i$.  By the definition of $F_i$, the graph $G_i$ is not $L_i$-colorable.
If $|L(x)|\ge\deg(x)$ for all $x\in C_i$ and $G_i$ has a $2$-connected component
that is neither complete nor an odd cycle, then Lemma~\ref{lemma-gallai} implies $|F_i|\le \deg_{G_i}(v)-1=n_i-1=c_i$.
Otherwise, by the assumptions $G_i-v$ has an $L_i$-coloring, and since this coloring cannot be extended to an $L_i$-coloring
of $G_i$, we conclude that $|F_i|\le \deg_{G_i}(v)=n_i=c_i$.
\end{proof}

\subsection{Reducible configurations}

When coloring from lists of size four, vertices of degree at most three can be colored greedily.
This argument no longer works in the flexibility setting, as this greedy coloring does not have any
freedom to satisfy requests.  However, the following weaker claim holds.

\begin{lemma}\label{lemma-redusmall}
If $G$ is a triangle-free graph, then a vertex of degree at most two, or two adjacent vertices of degree three,
form a $(1,4)$-reducible subgraph.
\end{lemma}
\begin{proof}
Suppose $v$ is a vertex of $G$ of degree at most two; then $\delta_{G,4}(v)\ge 2$.
Let $H$ be the subgraph of $G$ formed by $v$.  Then $(\deg_H+\delta_{G,4})(v)\ge 2$, and
$((\deg_H+\delta_{G,4})\downarrow v)(v)=1$ and $(\deg_H+\delta_{G,4}-1_{\{v\}})(v)\ge 1$.
Hence, $H$ is $L$-colorable whenever $L$ is a $(\deg_H+\delta_{G,4})\downarrow v$-assignment
or $(\deg_H+\delta_{G,4}-1_{\{v\}})$-assignment.  Consequently, $H$ is $(1,4)$-reducible.

Suppose now $v_1$ and $v_2$ are adjacent vertices of $G$ of degree three;
then $\delta_{G,4}(v_i)=1$ for $i\in\{1,2\}$.  Let $H$ be the subgraph of $G$ induced by $\{v_1,v_2\}$.
The function $f_i=(\deg_H+\delta_{G,4})\downarrow v_i$ satisfies $f_i(v_i)=1$ and $f_i(v_{3-i})=2$.
Note that the only non-empty $1$-independent sets in $H$ are $\{v_1\}$ and $\{v_2\}$,
and $(\deg_H+\delta_{G,4}-1_{\{v_i\}})=f_i$ for $i\in\{1,2\}$.
Clearly, $H$ is $L$-colorable whenever $L$ is an $f_1$-assignment or $f_2$-assignment,
and thus $H$ is $(1,4)$-reducible.
\end{proof}

\begin{figure}[h!]
\includegraphics[scale=0.7]{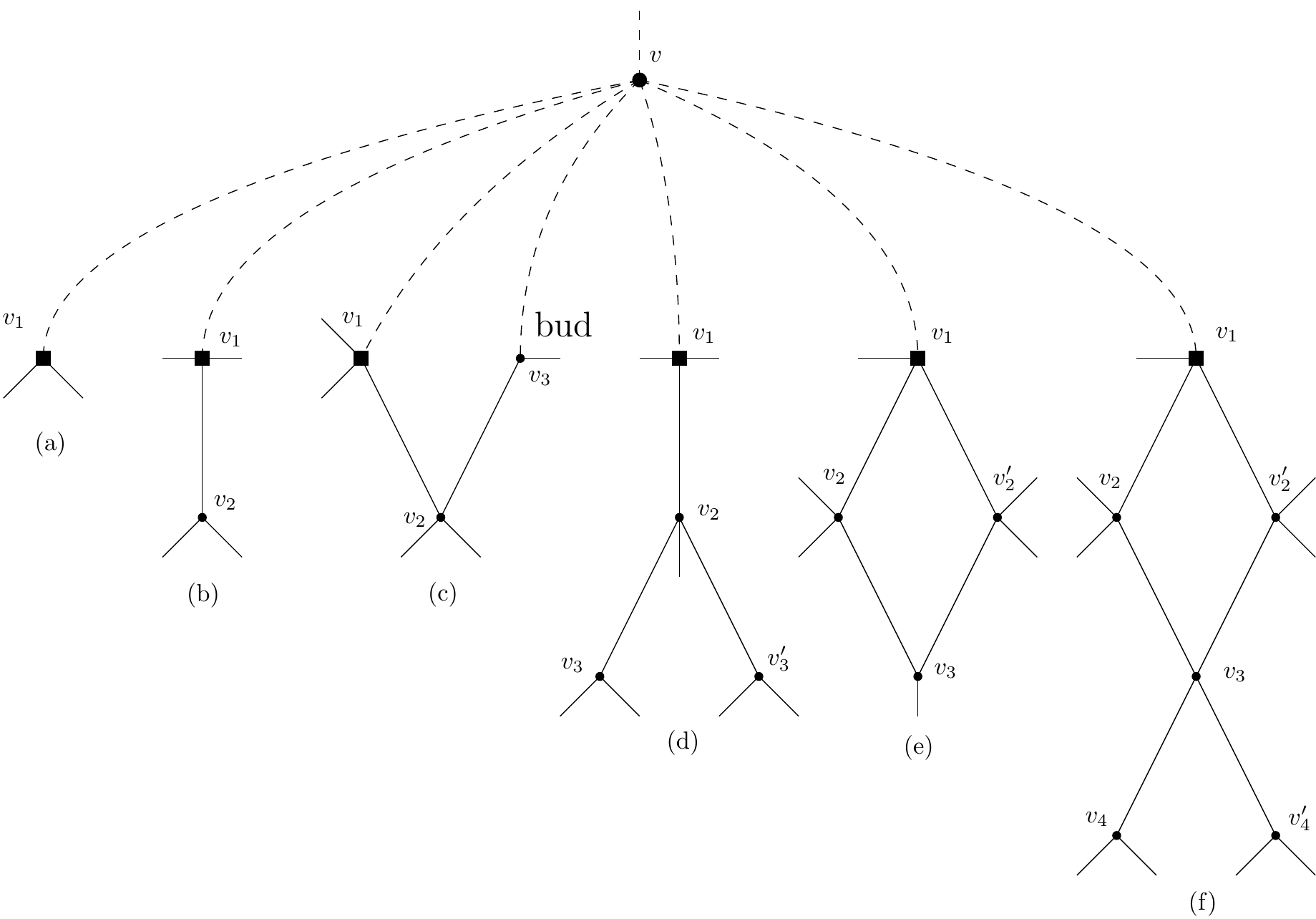}
\caption{Stalks.}\label{fig-stalks}
\end{figure}

We now describe a quite general class of $(1,4)$-reducible configurations.
Let $G$ be a triangle-free graph and $v$ a vertex of $G$. A \emph{$v$-stalk} is one of the following subgraphs (see Figure~\ref{fig-stalks}):
\begin{enumerate}
\item[(a)] A path $vv_1$ such that $\deg(v_1)=3$; or
\item[(b)] a path $vv_1v_2$ such that $\deg(v_1)=4$ and $\deg(v_2)=3$; or
\item[(c)] a cycle $vv_1v_2v_3$ such that $\deg(v_1)=\deg(v_2)=4$ and $\deg(v_3)=3$; or
\item[(d)] a path $vv_1v_2v_3$ and an edge $v_2v_3'$ with $v_3'\neq v_3$ such that $\deg(v_1)=\deg(v_2)=4$ and $\deg(v_3)=\deg(v'_3)=3$; or
\item[(e)] a path $vv_1v_2v_3$ and a path $v_1v'_2v_3$ with $v'_2\neq v_2$ such that $\deg(v_1)=\deg(v_2)=\deg(v'_2)=4$ and $\deg(v_3)=3$; or
\item[(f)] a path $vv_1v_2v_3v_4$, a path $v_1v'_2v_3$ with $v'_2\neq v_2$, and an edge $v_3v'_4$ with $v'_4\neq v_4$ such that $\deg(v_1)=\deg(v_2)=\deg(v'_2)=\deg(v_3)=4$ and $\deg(v_4)=\deg(v'_4)=3$.
\end{enumerate}
In all the cases, the \emph{root} of the stalk is the vertex $v_1$.  In the
case (c), we say that the vertex $v_3$ is the \emph{bud} of the stalk; in the
other cases, the stalk has no buds.  For a subgraph $C$ of $G$, a vertex $x$ is \emph{$(v,C)$-good} if $x$ is the root of a $v$-stalk
which is vertex-disjoint from $C$; in case the $v$-stalk has a bud $w$, we say that $x$ is \emph{$(v,C)$-good using the bud $w$}.

\begin{lemma}\label{lemma-mainredu}
Let $G$ be a plane triangle-free graph with the outer face bounded by a cycle $C$
such that each $(\le\!5)$-cycle in $G$ bounds a face.
Let $v$ be a vertex of $G$ of degree $d\ge 3$, not contained in $C$.  If $v$ has $d-1$ $(v,C)$-good neighbors, no two of them using the same bud,
then $G$ contains a $(1,4)$-reducible induced subgraph vertex-disjoint from $C$ with at most $6d-5$ vertices.
\end{lemma}
\begin{proof}
Let $X$ be a set of $d-1$ $(v,C)$-good neighbors of $v$, no two of them using the same bud, and for $x\in X$, let
$S_x$ be a $v$-stalk witnessing this is the case.
Let $H$ be the subgraph of $G$ induced by $\bigcup_{x\in X} V(S_x)$.
We clearly have $|V(H)|\le 6d-5$, and thus it suffices to show that $H$ is $(1,4)$-reducible.
This would be easy if the $v$-stalks were vertex-disjoint and there were no edges between them;
however, we need to argue about such overlaps.

Since each $(\le\!5)$-cycle in $G$ bounds a face and $C$ is vertex-disjoint from $H$,
the following claim (which excludes many of the overlaps) holds.
\begin{itemize}
\item[(\dag)] If $Q$ is a subgraph of $H$ and $Q$ is not a cycle, then $Q$ has a face
that is not bounded by a $(\le\!5)$-cycle.
\end{itemize}

For $x\in X$ and a vertex $w$ of $S_x$ denoted by $v_i$ or $v'_i$ in the definition of a $v$-stalk, let us define $\ell_x(w)=i$.
By (\dag) and the assumption that $G$ is triangle-free, we conclude that
$\ell_x(w)$ is equal to the distance between $v$ and $w$ in $G$, with the following exceptions:
\begin{itemize}
\item The vertex $v_3$ in case (c) is at distance $1$ from $v$,
\item the vertices $v_3$ and $v'_3$ in case (d) may be at distance $1$, $2$, or $3$ from $v$, and
\item the vertices $v_4$ and $v'_4$ in case (f) may be at distance $2$, $3$, or $4$ from $v$.
\end{itemize}
Suppose that $x,y\in X$ are distinct and $w\in V(S_x)\cap V(S_y)\setminus\{v\}$.
If both $\ell_x(w)$ and $\ell_y(w)$ are equal to the distance $r$ between $v$ and $w$,
then $r\ge 2$, as otherwise we would have $x=w=y$.  If say $\ell_y(w)$ is not equal to $r$,
then by the previous observation $\deg_G(w)=3$ and the stalk $S_y$ is defined according to cases (c), (d), or (f).
In the case (c), note furthermore that $S_x$ cannot be defined according to (c), as otherwise $x$ and $y$ would use the same bud.
In conclusion, the following claim (which we will refer to as ($\star$)) holds.
\begin{itemize}
\item $\ell_x(w)=\ell_y(w)$ is equal to the distance of $w$ from $v$ in $G$, and $w$ is not adjacent to $v$, or
\item $\deg_G(w)=3$ and at least one of the stalks $S_x$ and $S_y$ is defined according to cases (d) or (f), or
\item $\deg_G(w)=3$, $w$ is adjacent to $v$, one of $S_x$ and $S_y$ is defined according to (a) and the other one according to (c).
\end{itemize}
Note in particular that if a vertex $u\in V(H)\setminus\{v\}$ has degree $4$ in $G$
and $u$ is adjacent to $v$, then $u$ is the root of the stalk $S_u$ and $u$ does not belong to any other stalk.

Let $\delta=\delta_{G,4}$.  Let us first argue that $H$ satisfies (FIX).
\begin{subproof}
Consider any vertex $u\in V(H)$ and a $(\deg_H+\delta)\downarrow u$-assignment $L$ for $H$.
Let $c$ be a color in $L(u)$ and let $L'$ be the list assignment for $H-u$ obtained from $L$ by removing $c$
from the lists of neighbors of $u$.  We need to argue that $H-u$ is $L'$-colorable.

For all $z\in V(H-u)$, we have $|L'(z)|\ge \deg_{H-u}(z)+\delta(z)=\deg_{H-u}(z)+4-\deg_G(z)$.
Hence, if $z\neq v$ then $|L'(z)|\ge \deg_{H-u}(z)$, and if $\deg_G(z)=3$, then $|L'(z)|>\deg_{H-u}(z)$.
If $u=v$, then note that by the definition of a $v$-stalk, each component of $H-v$ contains a vertex whose degree in $G$ is three,
and we conclude that $H-u$ is $L'$-colorable by applying Lemma~\ref{lemma-gallai} to each component.  Hence, suppose that $u\neq v$.

Note that the definition of a $v$-stalk
ensures that $v$ has a neighbor in each component of $H-\{u,v\}$ that contains only vertices whose degree in $G$ is four.
A neighbor $z$ of $v$ in $H$ is \emph{dangerous}
if either $z=u$ or $z$ belongs to a component of $H-\{u,v\}$ that contains only vertices of degree $4$.
If at most $2$ neighbors of $v$ are dangerous, then first greedily $L'$-color the components of $H-\{u,v\}$ containing dangerous
vertices (this is possible, since these dangerous vertices are adjacent to the vertex $v$ which has not been colored yet), then give a color from $L'(v)$
to $v$ (which is possible, since $|L(v)|\ge 3$ is greater than the number of dangerous neighbors of $v$),
and finally extend the coloring to the remaining components of $H-\{u,v\}$
(which is possible, since each such component contains a vertex $z$ of degree three, which satisfies $|L'(z)|>\deg_{H-u}(z)$).
Thus, to prove that $H-u$ is $L'$-colorable, it suffices to argue that $v$ has at most two dangerous neighbors.

Suppose for a contradiction that $x$, $y$, and $q$ are distinct dangerous neighbors of $v$. Without loss of generality, $x\neq u\neq y$, and thus the components of
$H-\{u,v\}$ containing $x$ and $y$ only consist of vertices of degree $4$.
Since $\deg(x)=\deg(y)=4$ and $x$ and $y$ are adjacent to $v$, ($\star$) implies that $x$ and $y$ are the roots of stalks $S_x$ and $S_y$.
We conclude that $u\in V(S_x)\cap V(S_y)\setminus \{v\}$
and the removal of $u$ separates $x$ and $y$ from the vertices of degree $3$ in their stalks (possibly, $u$ is this vertex of degree $3$).
If $\deg(u)=3$, this implies that neither of the stalks $S_x$ and $S_y$ is defined according to (d) or (f), and clearly neither
is defined by (a).  By ($\star$), we conclude that $\ell_x(u)$ and $\ell_y(u)$ are equal to the distance $r\ge 2$ between $u$ and $v$ in $G$,
and $u$ is not a neighbor of $v$.  The same claim is implied by ($\star$) when $\deg(u)=4$.

Consequently $q\neq u$, and thus $q$ is the root of a stalk $S_q$, $u\in V(S_q)\setminus\{v\}$, and $\ell_q(u)=r$.
If $r=2$, then $H$ contains three paths
of length two between $v$ and $u$, contradicting (\dag).
Hence, $r\ge 3$.  Since $u$ is non-adjacent to $v$, together with the observation
that if $\deg(u)=3$ then the stalks $S_x$, $S_y$, and $S_q$ are not defined according to (d) and (f), this excludes the possibilities that the stalks
are defined according to the cases (a), (b), (c), (d), and when $\deg(u)=3$ also (f).  Hence,
either $\deg(u)=3$ and the stalks $S_x$, $S_y$, and $S_q$ are defined according to (e), or $\deg(u)=4$, $r=3$ and
the stalks are defined according to (f). However, it is easy to see that it is not possible to arrange the three stalks with
distinct roots in the plane without exceeding the degree of $u$ or violating the condition (\dag). This is a contradiction,
finishing the argument that (FIX) holds.
\end{subproof}

Next, let us argue that $H$ satisfies (FORB).
\begin{subproof}
Let $I$ be a $1$-independent set in $H$ of size at most $2$, and let $L$ be a $(\deg_H+\delta-1_I)$-assignment for $H$.
We need to argue that $H$ is $L$-colorable.  Let $H'$ be the induced subgraph of $H$ obtained by initializing $H'\colonequals H$
and repeatedly performing the following reductions as long as possible:
\begin{itemize}
\item If $u\in V(H')$ satisfies $|L(u)|>\deg_{H'}(u)$, then let $H'\colonequals H'-u$.
\item If $K\subseteq V(H')$ induces a $2$-connected subgraph, neither a clique nor an odd cycle, and $|L(u)|\ge \deg_{H'}(u)$ for all $x\in K$,
then let $H'\colonequals H'-K$.
\end{itemize}
Note that $H'$ is uniquely determined, regardless of the choice of removed subgraphs, since the degrees in $H'$ do not increase,
and if a part of a set $K$ satisfying the assumptions of the second reduction gets removed by another reduction,
then all vertices of $K$ will eventually be removed because of the first reduction.
Using Lemma~\ref{lemma-gallai}, it is easy to see that $H$ is $L$-colorable if and only if $H'$ is $L$-colorable.

\begin{figure}
\includegraphics[scale=0.65]{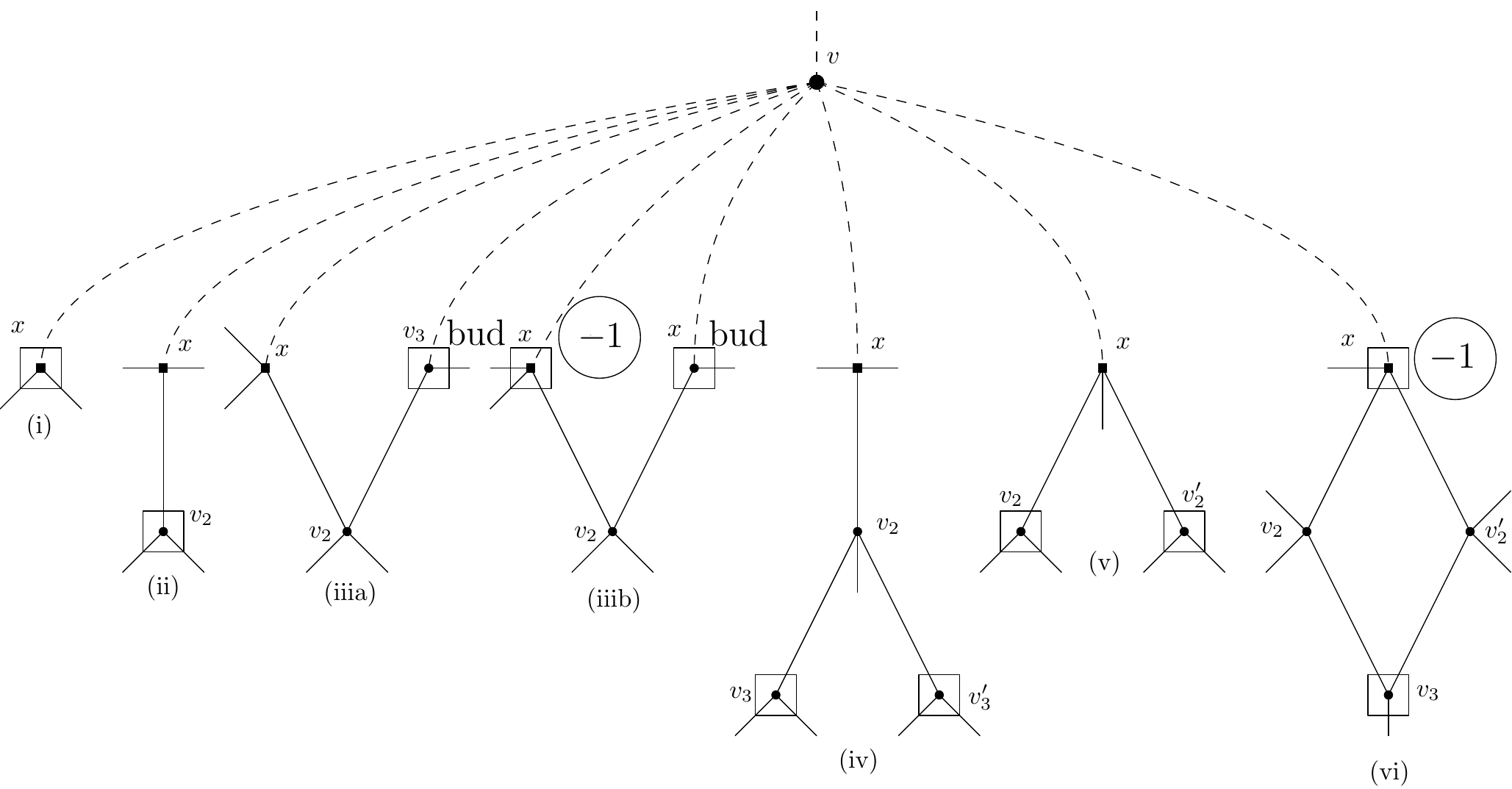}
\caption{Reduced stalks. Vertices in $I$ are marked by squares. Vertices $u\neq v$ with $|L(u)|=\deg_{H'}(u)-1$ are marked by $-1$.}\label{fig-redstalk}
\end{figure}

Note that for $u\in V(H)\setminus \{v\}$, we have $|L(u)|>\deg_H(u)$ if $\deg_G(u)=3$ and $u\not\in I$,
$|L(u)|\ge\deg_H(u)$ if $\deg_G(u)+1_I(u)=4$, and $|L(u)|\ge\deg_H(u)-1$ if $\deg_G(u)=4$ and $u\in I$.
Consider a vertex $x\in X$ and the $v$-stalk $S_x$.  A straightforward case analysis shows that if $v\in V(H')$,
then either $S_x\cap H'=v$, or $S_x\cap H'$ satisfies one of the following conditions (see Figure~\ref{fig-redstalk}).
\begin{itemize}
\item[(i)] $S_x\cap H'$ consists of the edge $vx$, $x\in I$, and $|L(x)|=\deg_{H'}(x)$; or,
\item[(ii)] $S_x\cap H'$ is a path $vxv_2$, $\deg_G(x)=4$, $v_2\in I$, $|L(v_2)|=\deg_{H'}(v_2)$, and $|L(x)|=\deg_{H'}(x)$; or,
\item[(iii)] $S_x\cap H'$ is a $4$-cycle $vxv_2v_3$, the $v$-stalk $S_x$ satisfies (c) and $v_3$ is its bud,
$v_3\in I$, $|L(v_2)|=\deg_{H'}(v_2)$, $|L(v_3)|=\deg_{H'}(v_3)$,
and either
\begin{itemize}
\item[(iiia)] $|L(x)|=\deg_{H'}(x)$, or
\item[(iiib)] $x\in I$ and $|L(x)|=\deg_{H'}(x)-1$;
\end{itemize}
or,
\item[(iv)] $S_x\cap H'$ consists of a path $vxv_2v_3$ and an edge $v_2v'_3$,
$v_3,v'_3\in I$, $\deg_G(x)=\deg_G(v_2)=4$, $|L(x)|=\deg_{H'}(x)$, $|L(v_2)|=\deg_{H'}(v_2)$, $|L(v_3)|=\deg_{H'}(v_3)$, and
$|L(v'_3)|=\deg_{H'}(v'_3)$; or,
\item[(v)] $S_x\cap H'$ consists of a path $vxv_2$ and an edge $xv'_2$,
$v_2,v'_2\in I$, $\deg_G(x)=\deg_G(v_2)=\deg_G(v'_2)=4$, $|L(x)|=\deg_{H'}(x)$, $|L(v_2)|=\deg_{H'}(v_2)$, and $|L(v'_2)|=\deg_{H'}(v'_2)$; or,
\item[(vi)] $S_x\cap H'$ consists of an edge $vx$ and a $4$-cycle $xv_2v_3v'_2$,
$x,v_3\in I$, $\deg_G(x)=\deg_G(v_2)=\deg_G(v'_2)=4$, $\deg_G(v_3)=3$, $|L(v_2)|=\deg_{H'}(v_2)$, $|L(v'_2)|=\deg_{H'}(v'_2)$,
$|L(v_3)|=\deg_{H'}(v_3)$, and $|L(x)|=\deg_{H'}(x)-1$.
\end{itemize}
The same case analysis also shows that if $v\not\in V(H')$, then $V(H')=\emptyset$, and thus $H'$ and $H$ are $L$-colorable.
Hence, we can assume that $v\in V(H')$.  Analogously, we conclude that $H'-v$ is $L$-colorable.
Note also that $S_x\cap H'$ is an induced subgraph of $H'$, except possibly for the
case (iv), where $v_3$ or $v'_3$ can be adjacent to $v$ (but not both by (\dag)).

Furthermore, if $z\in V(H')\setminus\{v\}$ satisfies $|L(z)|<\deg_{H'}(z)$, then
$z$ is the root of a $v$-stalk $S_z$ such that $S_z\cap H'$ satisfies (iiib) or (vi), and in particular $S_z\cap H'$ contains
the other vertex $z'$ of $I$, and $z'$ satisfies $|L(z')|=\deg_{H'}(z')$.  Consequently, $|L(z)|<\deg_{H'}(z)$ for at most one vertex $z\in V(H')\setminus\{v\}$.

Let $C_1$, \ldots, $C_k$ be the vertex sets of the components of $H'-v$, and let $c_1$, \ldots, $c_k$ be the corresponding integers defined in the statement Corollary~\ref{cor-rem}.
Each of the components contains a vertex of $I$, and thus $k\le |I\setminus\{v\}|$.
We now argue that $c_1+\ldots+c_k<|L(v)|$, implying that $H'$ is $L$-colorable by Corollary~\ref{cor-rem}, and thus $H$ is $L$-colorable as well.
Let $X'=\{x\in X:S_x\cap H'\neq v\}$.

Let us first consider the case that $C_1\cap I$ consists of exactly one vertex $z$.
If $V(S_x)\cap C_1\neq \emptyset$ for some $x\in X'$, then $S_x\cap H'$
satisfies (i), (ii), or (iiia).
If $z$ is adjacent to $v$, then all such $v$-stalks satisfy (i) or (iiia), and there is at most one such $v$-stalk satisfying (iiia)
since no two $v$-stalks satisfying (c) use the same bud.  Hence, $H'[C_1\cup\{v\}]$ is either an edge or a $4$-cycle,
and $c_1=1$.  If $z$ is at distance $2$ from $v$ in $H'$, then the $v$-stalks intersecting $C_1$ satisfy (ii),
and there are at most two such $v$-stalks sharing the vertex $z$ by (\dag).  Hence, $c_1=1$ again.
Note that $|L(v)|\ge 3-1_I(v)\ge 2$, and thus if $k=1$, then $c_1+\ldots+c_k=c_1=1<|L(v)|$.
If $k=2$, then $v\not\in I$ since $|I|\le 2$ and $C_2\cap I\neq\emptyset$, and thus $|L(v)|\ge 3$; furthermore,
the same argument shows $c_2=1$, and consequently $c_1+\ldots+c_k=c_1+c_2=2<|L(v)|$.

Therefore, we can assume that $C_1$ contains two vertices of $I$, and since $|I|\le 2$, we have $k=1$ and $v\not\in I$; hence, $|L(v)|\ge 3$,
and it suffices to argue that $c_1\le 2$.

If there exists a vertex $z\in V(H')\setminus \{v\}$ with $|L(z)|<\deg_{H'}(z)$, then $z$ is the root of the $v$-stalk $S_z$ and
$S_z\cap H'$ satisfies (iiib) or (vi).  The root $z$ is in a unique stalk by ($\star$), so there are no $v$-stalks
satisfying (iv) or (v), or a $v$-stalk satisfying (iiib) or (vi) other than $S_z$. The vertex of $I$ other than $z$ cannot belong to a stalk
satisfying (ii) or (iiia) by the absence of triangles, (\dag), and the assumption that no two $v$-stalks use the same bud.  Consequently, $H'=S_z\cap H'$, and thus $c_1\le 2$.

Hence, we can assume that $|L(z)|=\deg_{H'}(z)$ for all $z\in V(H')\setminus \{v\}$, and in particular no $v$-stalk satisfies (iiib) or (vi).
Let us now consider the case that $S_x\cap H'$ satisfies (v) for some $x\in X'$.  We cannot have another $v$-stalk satisfying
(iv) or (v), since then $H'-v$ would contain a $4$-cycle that should be removed by the second reduction rule.
Furthermore, the vertices of $I$ are not adjacent to $v$ in this case, and thus 
$S_y\cap H'$ satisfies (ii) for all vertices $y\in X'\setminus\{x\}$.  By (\dag), we conclude that $|X'|\le 3$.  Note also that if $|X'|=3$, then $c_1=|X'|-1$.
Hence, $c_1\le 2$.  Therefore, we can assume that no $v$-stalk satisfies (v).

Suppose now that $S_x\cap H'$ satisfies (iv) for some $x\in X'$.  If $X'$ contains another vertex $y$ with this property,
then since $H'-v$ does not contain a $4$-cycle (which would be removed by the second reduction rule),
we conclude that $(S_x\cup S_y)\cap H'$ consist of a $4$-cycle $vxv_2y$ and vertices $v_3, v'_3\in I$ adjacent to $v_2$,
and by (\dag), neither $v_3$ nor $v'_3$ is adjacent to $v$.
In this case, (\dag) implies that $H'=(S_x\cup S_y)\cap H'$,
and thus $c_1=1$.  Hence, we can assume that $X'\setminus\{x\}$ does not contain any vertex $y$ such that $S_y\cap H'$ satisfies (iv).
If neither of the vertices of $I$ is adjacent to $v$, then $S_y\cap H'$ satisfies (ii) for all $y\in X\setminus\{x\}$,
and $|X'|\le 3$ by (\dag), and $c_1=|X'|-1$ if $|X'|=3$.  Hence, $c_1\le 2$.  Finally, let us consider the case that
a vertex $z\in I$ is adjacent to $v$.  Then the other vertex of $I$ is not contained in another $v$-stalk by (\dag),
and $z$ can be contained in at most one $v$-stalk satisfying (iiia); hence $c_1\le 2$.

Therefore, we can assume that $S_x\cap H'$ satisfies (i), (ii), or (iiia) for every $x\in X'$.
Suppose that all vertices of $I$ are adjacent to $v$, and thus $S_x\cap H'$ satisfies (i) or (iiia) for every $x\in X'$.
By (\dag) and the assumption that no two $v$-stalks use the same bud,
either $H'=S_x\cap H'$ for some $x\in X'$, or there exist distinct $x,y\in X'$
such that $S_x\cap S_y\cap H'=v$.
In the former case, we have $c_1=1$.  In the latter case, since $H'[C_1]$ is connected,
we can also assume there is an edge between $S_x\cap H'-v$ and $S_y\cap H'-v$;
however, this is not possible, since $H'$ is triangle-free and satisfies (\dag).

Finally, suppose that there exists a vertex $z\in I$ non-adjacent to $v$, which necessarily is contained
in a stalk $S_x$ such that $S_x\cap H'$ satisfies (ii).  Let $X'_z=\{x\in X':z\in V(S_x\cap H')\}$ and let $H'_z=\bigcup_{x\in X'_z} S_x\cap H'$.
By (\dag), $|X'_z|\le 2$, and since $G$ is triangle-free, we have $S_y\cap H'_z=v$ for all $y\in X'\setminus X'_z$.
Since $|C_1\cap I|=2$, there exists $y\in X'\setminus X'_z$ such that $z\not\in V(S_y)$.  Since $H'[C_1]$ is
connected, we can furthermore choose $y$ so that there is an edge $e$ between $H'_z-v$ and $S_y\cap H'-v$.
Since $H'$ is triangle-free and $I$ is an independent set, $S_y\cap H'$ does not satisfy (i), and thus it satisfies (ii) or (iiia).
If it satisfies (iiia), then (\dag) implies that $|X'_z|=1$, $e=yz$ and $H'=H'_z\cup (S_y\cap H')+e$, and $c_1=2$.
If $S_y\cap H'$ satisfies (ii), then (\dag) similarly implies that $|X'|\le 3$, $S_y\cap H'$ satisfies (ii) for all $y\in X'$,
and $H'=\big(\bigcup_{y\in X'} S_y\cap H'\big)+e$, and $c_1\le 2$.

Hence, in all the cases, $H$ is $L$-colorable, and (FORB) holds.
\end{subproof}
Since $H$ satisfies both (FIX) and (FORB), it forms a $(1,4)$-reducible induced subgraph of $G$.
\end{proof}

\begin{figure}
\begin{center}
\includegraphics[scale=0.7]{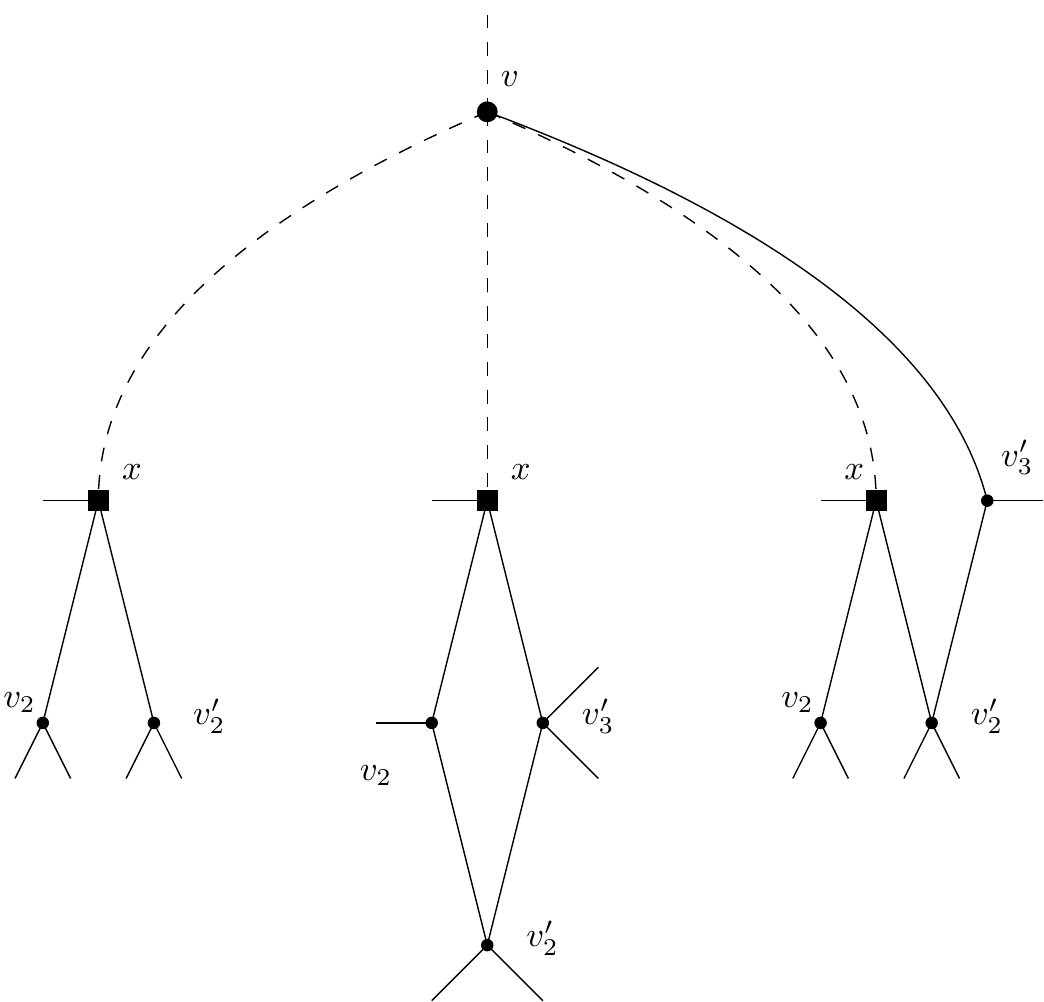}
\end{center}
\caption{Excellent stalks extending a (b) stalk.}\label{fig-extstalk}
\end{figure}

Let $G$ be a triangle-free graph, let $C$ be a subgraph of $G$, and let $v$ be a vertex of $G$.
We say that a neighbor $x$ of $v$ is \emph{$(v,C)$-excellent} if one of the following conditions holds.
\begin{itemize}
\item $x$ is $(v,C)$-good due to a $v$-stalk satisfying
(a), (d), (e), or (f), or
\item $x$ is $(v,C)$-good due to a $v$-stalk satisfying (b) such that (see Figure~\ref{fig-extstalk})
\begin{itemize}
\item $x=v_1$ has a neighbor $v'_2\neq v_2$ of degree $3$ not belonging to $C$, or
\item there exists a $4$-cycle $v_1v_2v'_3v'_2$ (where $x=v_1$) with $v'_2,v'_3\not\in V(C)$ and $\deg(v'_2)=\deg(v'_3)=4$, or
\item there exists a $4$-cycle $v_1v'_2v'_3v$ (where $x=v_1$) with $v'_2,v'_3\not\in V(C)$, $\deg(v'_2)=4$ and $\deg(v'_3)=3$.
\end{itemize}
\end{itemize}
The \emph{extended $v$-stalk} of a $(v,C)$-excellent vertex is its $v$-stalk, together with the edge $v_1v'_2$, the
path $v_2v'_3v'_2v_1$, or the path $v_1v'_2v'_3v$ if the $v$-stalk satisfies one of the last three cases, respectively.

\begin{figure}
\begin{center}
\includegraphics[scale=0.7]{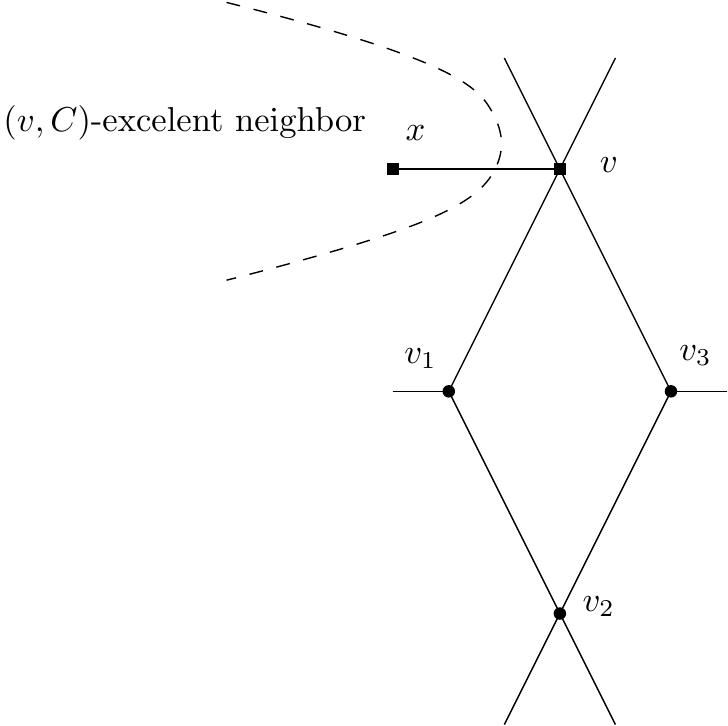}
\end{center}
\caption{Situation in Lemma~\ref{lemma-5redu}.}\label{fig-l8}
\end{figure}

We now give another class of reducible configurations at vertex of degree five.

\begin{lemma}\label{lemma-5redu}
Let $G$ be a plane triangle-free graph with the outer face bounded by a cycle $C$
such that each $(\le\!5)$-cycle in $G$ bounds a face.
Let $v$ be a vertex of $G$ of degree $5$ incident with a $4$-face bounded by a cycle
$vv_1v_2v_3$ such that $\deg(v_1)=\deg(v_3)=3$, $\deg(v_2)=4$, and $v,v_1,v_2,v_3\not\in V(C)$.
If $v$ has a $(v,C)$-excellent neighbor $x$ distinct from $v_1$ and $v_3$, then $G$ contains
a $(1,4)$-reducible induced subgraph with at most $10$ vertices, vertex-disjoint from $C$.
\end{lemma}
\begin{proof}
Let $S_x\subseteq S'_x$ be the $v$-stalk and the extended $v$-stalk with root $x$, respectively.
Let $H$ be the subgraph of $G$ induced by $\{v,v_1,v_2,v_3\}\cup V(S'_x)$.
Clearly, $|V(H)|\le 10$, and thus it suffices to prove that $H$ is $(1,4)$-reducible.

Let $\delta=\delta_{G,4}$.  Let us first argue that $H$ satisfies (FIX).
\begin{subproof}
Consider a vertex $u\in V(H)$ and a $(\deg_H+\delta)\downarrow u$-assignment $L$ for $H$.
Let $c$ be a color in $L(u)$ and let $L'$ be the list assignment for $H-u$ obtained from $L$ by removing $c$
from the lists of neighbors of $u$.  We need to argue that $H-u$ is $L'$-colorable.
Note that $|L'(z)|\ge \deg_{H-u}(z)$ for all $z\in V(H-u)\setminus\{v\}$, and $|L'(z)|>\deg_{H-u}(z)$ if $\deg_G(z)=3$.
By the definition of an extended $v$-stalk and the assumptions on the degrees of $v_1$ and $v_3$,
each component of $H-v$ contains a vertex whose degree in $G$ is three, and thus if $u=v$, then $H-u$ is $L'$-colorable.  Hence, suppose that $u\neq v$.

Note that the definition of an extended $v$-stalk
ensures that $v$ has a neighbor in each component of $H-\{u,v\}$ that contains only vertices whose degree in $G$ is four.
Recall a neighbor $z$ of $v$ in $H$ is \emph{dangerous}
if either $z=u$ or $z$ belongs to a component of $H-\{u,v\}$ that contains only vertices of degree $4$.
Inspecting all possible extended $v$-stalks $S'_x$, we conclude using (\dag) and the assumption $G$ is triangle-free that
$v$ has no neighbors in $H$ other than $x$ whose degree in $G$ is four.  Consequently, $v$ cannot have dangerous neighbors other than $x$ and $u$.
Furthermore, the same inspection shows that if $u\neq x$ and $u$ is adjacent to $v$, then $S'_x-u$ contains a path from $x$ to a vertex whose degree in $G$ is three, and thus $x$ is not dangerous.

Therefore, $v$ has at most one dangerous neighbor.  We first greedily $L'$-color the component of $H-\{u,v\}$ containing the dangerous
vertex, if any (this is possible, since the dangerous vertex is adjacent to the vertex $v$ which has not been colored yet), then give a color from $L'(v)$
to $v$ (which is possible, since $|L(v)|\ge 2$ is greater than the number of dangerous neighbors of $v$),
and finally extend the coloring to the remaining components of $H-\{u,v\}$
(which is possible, since each such component contains a vertex whose degree in $G$ is three).
Therefore, $H-u$ is $L'$-colorable, implying (FIX).
\end{subproof}

Next, we show $H$ satisfies (FORB).
\begin{subproof}
Let $I$ be a $1$-independent set in $H$ of size at most $2$, and let $L$ be a $(\deg_H+\delta-1_I)$-assignment for $H$.
We need to argue that $H$ is $L$-colorable.  Let $H'$ be the induced subgraph of $H$ obtained by the same reduction rules
as in the proof of Lemma~\ref{lemma-mainredu}; it suffices to prove that $H'$ is $L$-colorable.

Suppose first that at least one of $v_1$ or $v_3$ does not belong to $I$, and thus $v_1$, $v_2$, and $v_3$
get removed according to the first reduction rule.   If $v\not\in I$, then $v$ also gets removed by the first
reduction rule, and then it is easy to see that $H'$ is empty.  If $v\in I$, then $S_x\cap H'$ is
not described by (i), (iii), (iv), (v), and (vi) from the proof of Lemma~\ref{lemma-mainredu},
since $I$ is an independent set and $|I\setminus\{v\}|\le 1$;
we conclude that $H'$ is empty unless $S_x$ satisfies (b).  But in that case, the inspection of the extended $v$-stalks
shows that $S'_x\cap H'=\emptyset$.  In all the cases, we conclude that $H'$ is empty, and thus $H$ is $L$-colorable.

Hence, we can assume that $I=\{v_1,v_3\}$.  In particular, both vertices of $I$ are adjacent to $v$
and have degree three in $G$, and since $G$ is triangle-free and satisfies (\dag),
we conclude that $S_x\cap H'$ cannot satisfy any of the cases (i) to (vi).  This implies $S'_x\cap H'\subseteq v$.
It follows that the $4$-cycle $vv_1v_2v_3$ also gets reduced and $H'$ is empty.
Therefore, $H$ is $L$-colorable, implying (FORB).
\end{subproof}
Hence, both (FIX) and (FORB) hold, and thus $H$ is $(1,4)$-reducible.
\end{proof}

\begin{figure}
\begin{center}
\includegraphics[scale=0.7]{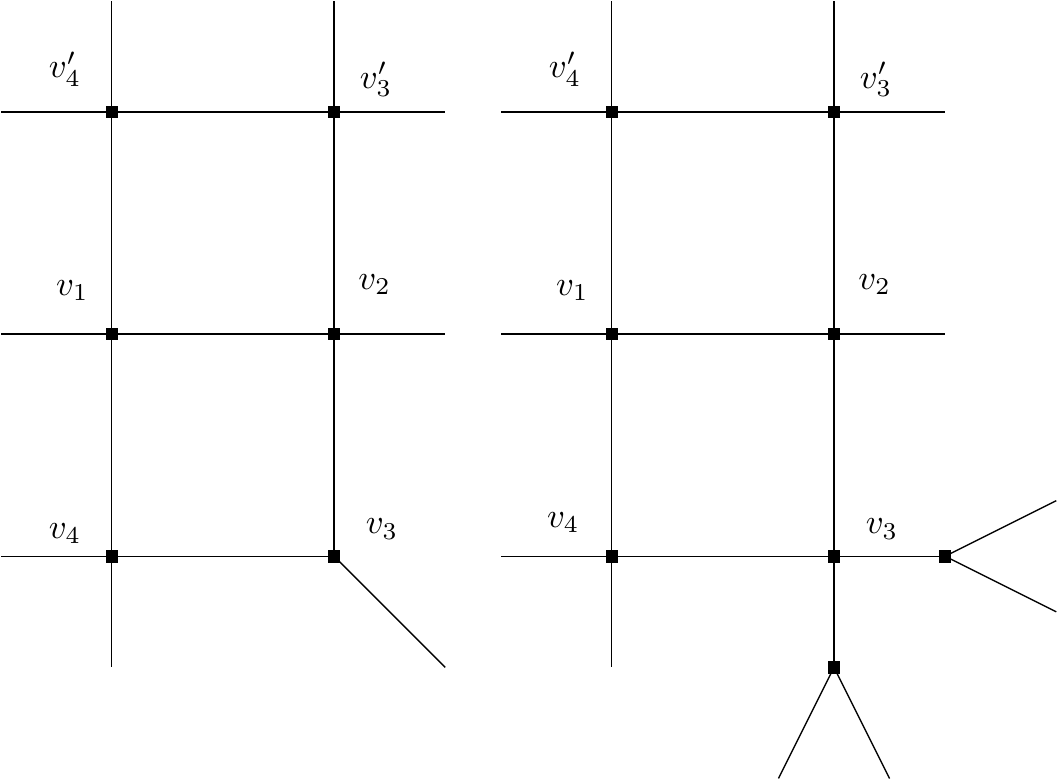}
\end{center}
\caption{Situations in Lemma~\ref{lemma-spec4}.}\label{fig-l10}
\end{figure}

Finally, let us describe reducible configurations at adjacent 4-faces.

\begin{lemma}\label{lemma-spec4}
Let $G$ be a plane triangle-free graph with the outer face bounded by a cycle $C$
such that each $(\le\!5)$-cycle in $G$ bounds a face.
Let $v_1v_2v_3v_4$ and $v_1v_2v_3'v_4'$ be cycles bounding distinct $4$-faces
(sharing the edge $v_1v_2$), vertex-disjoint from $C$, and such that $\deg(v_1)=\deg(v_2)=\deg(v_3)=\deg(v'_3)=\deg(v'_4)=4$.
If either $\deg(v_3)=3$, or $\deg(v_3)=4$ and $v_3$ has two neighbors of degree three not belonging to $V(C)$,
then $G$ contains a $(1,4)$-reducible induced subgraph with at most $8$ vertices, vertex-disjoint from $C$.
\end{lemma}
\begin{proof}
If $\deg(v_3)=4$, then let $Z$ consist of the two neighbors of $v_3$ of degree three, otherwise let $Z=\emptyset$.
Let $H$ be the subgraph of $G$ induced by $\{v_1,v_2,v_3,v_4,v_3',v'_4\}\cup Z$.
Clearly, $|V(H)|\le 8$, and thus it suffices to prove that $H$ is $(1,4)$-reducible.
Note that $v_3v'_4,v'_3v_4\not\in E(G)$ by (\dag).

Let $\delta=\delta_{G,4}$.  Let us first argue that $H$ satisfies (FIX).
Consider a vertex $u\in V(H)$ and a $(\deg_H+\delta)\downarrow u$-assignment $L$ for $H$.
Let $c$ be a color in $L(u)$ and let $L'$ be the list assignment for $H-u$ obtained from $L$ by removing $c$
from the lists of neighbors of $u$.  We need to argue that $H-u$ is $L'$-colorable.
Note that $|L'(z)|\ge \deg_{H-u}(z)$ for all $z\in V(H-u)$, and $|L'(z)|>\deg_{H-u}(z)$ if $\deg_G(z)=3$.

If $u\neq v_3$ then $H-u$ is connected and contains a vertex $z$ such that $\deg_{H-u}(z)<|L(z)|$, and thus $H-u$ is $L'$-colorable.
On the other hand, each component of $H-v_3$ contains either a vertex whose degree in $G$ is three or a $4$-cycle,
and thus if $u=v_3$, then $H-v_3$ is $L'$-colorable by Lemma~\ref{lemma-gallai}.  Hence, (FIX) holds.

Next, we show that $H$ satisfies (FORB).  Let  $I$ be a $1$-independent set in $H$ of size at most $2$, and let $L$ be a $(\deg_H+\delta-1_I)$-assignment for $H$.
We need to argue that $H$ is $L$-colorable.  Let $H'$ be the induced subgraph of $H$ obtained by the same reduction rules as in the proof
of Lemma~\ref{lemma-mainredu}; it suffices to prove that $H'$ is $L$-colorable.

If $I\cap \{v_1,v_2,v'_3,v'_4\}=\emptyset$, then the 4-cycle $v_1v_2v'_3v'_4$ is removed by the second rule.  Then,
at most one of $v_3$ and $v_4$ belongs to $I$, and thus both get removed by the first rule in turn.  Finally, vertices of $Z$ are
removed by the first rule.  Hence, $H'$ is empty, and thus $L$-colorable.  Therefore, we can assume that $I\cap \{v_1,v_2,v'_3,v'_4\}\neq\emptyset$

If $Z$, $v_3$, and $v_4$ are removed (in any order) during the construction of $H'$, then since at most one of $v_1$ and $v_2$ belongs
to $I$, the vertices $v_1$ and $v_2$ are next removed by the first rule.  Finally, since at most one of $v'_3$ and $v'_4$ belongs to $I$,
the vertices $v'_3$ and $v'_4$ are removed by the first rule.  Hence, $H'$ is empty, and thus $L$-colorable.

Therefore, we can assume that $V(H')\cap (\{v_3,v_4\}\cup Z)\neq \emptyset$, and in particular $I\cap (\{v_3,v_4\}\cup Z)\neq\emptyset$.
Consequently, $|I\cap \{v_1,v_2,v'_3,v'_4\}|=1$ and $|I\cap (\{v_3,v_4\}\cup Z)|=1$.  Furthermore, a straightforward case analysis shows
that either $V(H')\cap (\{v_3,v_4\}\cup Z)=\{v_4\}$ and $v_4\in I$, or $\deg(v_3)=3$ and $v_3\in I$.
In the former case, if $v_2\in I$ then the 4-cycle $v_1v_2v'_3v'_4$ is removed by the second rule; if $v_2\not\in I$, then $v_2$, $v_1$, $v'_3$, and $v'_4$
are removed by the first rule in some order.  In either case $v_4$ is then removed by the first rule, which is a contradiction.

Hence, we can assume $\deg(v_3)=3$ and $v_3\in I$.  Furthermore, since the $4$-cycle $v_1v_2v_3v_4$ is not removed by the second rule,
we conclude $v_1\in I$.  Note that $|L(v_1)|,|L(v_3)|,|L(v_4)|,|L(v'_3)|,|L(v'_4)|\ge 2$ and $|L(v_2)|\ge 3$.
If $L(v_1)\cap L(v_3)\neq\emptyset$, we give $v_1$ and $v_3$ the same color and then greedily $L$-color $v_4$, $v'_4$, $v'_3$, and $v_2$ in order.
Otherwise, for some $i\in \{1,3\}$, there exists a color $c\in L(v_i)$ such that $|L(v_2)\setminus \{c\}|\ge 3$.
Give $v_i$ the color $c$ and then greedily $L$-color $v_4$, $v_{4-i}$, $v'_4$, $v'_3$, and $v_2$ in order.
In both cases, we obtain an $L$-coloring of $H$, implying (FORB).

Hence, both (FIX) and (FORB) hold, and thus $H$ is $(1,4)$-reducible.
\end{proof}

\subsection{Discharging}\label{sec-discharging}

Let $G_0$ be a connected plane triangle-free graph of minimum degree at least three.
Note that $G_0$ has a face of length at most $5$, bounded by a cycle $C_0$.
Without loss of generality, we can assume that $C_0$ bounds the outer face of $G_0$.
Let $C$ be a $(\le\!5)$-cycle in $G_0$ such that the open disk $\Delta_C$ bounded by $C$ is not a face of $G_0$,
and $\Delta_C$ is minimal among the cycles $C$ with this property.  Let $G$ be the subgraph of $G_0$ drawn in the closure of $\Delta_C$.
By the choice of $C$, every $(\le\!5)$-cycle in $G$ bounds a face.  Note also that if $H$ is a $(1,4)$-reducible induced subgraph
of $G$ disjoint from $C$, then $H$ is also a $(1,4)$-reducible induced subgraph of $G_0$.

Let us assign charge $\ch_0(v)=\deg(v) - 4$ to each vertex $v\in V(G)\setminus V(C)$,
charge $\ch_0(v)=\deg(v)-7/3$ to each vertex $v\in V(C)$, charge $\ch_0(f)=|f|-4$ to each face $f$ of $G$ distinct from the outer one,
and $\ch_0(f_0)=0$ to the outer face $f_0$,
where $|f|$ denotes the length of the facial walk of $f$.  By Euler's formula, we have
$\sum_{v\in V(G)} \ch_0(v)+\sum_{f\in F(G)}\ch_0(f)=(2|E(G)|-4|V(G)|)+(2|E(G)|-4|F(G)|)+\tfrac{5}{3}|C|-(|C|-4)=4(|E(G)|-|V(G)|-|F(G)|)+\tfrac{2}{3}|C|+4=-4+\tfrac{2}{3}|C|<0.$
Let $f$ be a face of $G$ and let $W$ be its facial walk in the clockwise order around $f$.  An \emph{angle} of $f$ is a subwalk of $W$ of length two, and the \emph{tip} of the
angle is the central vertex of this subwalk.  Note that $f$ has exactly $|f|$ angles, while it may be incident with a smaller number of vertices if $G$ is not $2$-connected.
A $4$-face $f$ is \emph{poor} if all vertices incident with $f$ have degree at most $4$ and do not belong to $V(C)$,
and one of the incident vertices either has degree three or two neighbors of degree three not belonging to $V(C)$.

Let us now redistribute the charge according to the following rules:
\begin{itemize}
\item[(R0)] For each non-outer face $f$ of $G$ and for each angle of $f$ with tip $v$, if either $v\not\in V(C)$ and $v$ has degree three,
or $v\in V(C)$ and $v$ has degree two, then $f$ sends $1/3$ to $v$.
\item[(R1)] Let $f$ be a non-outer face of $G$ and let $v_1vv_2$ be an angle of $f$ such that $v\not\in V(C)$ and $\deg(v)=4$.  Let $v_3$ and $v_4$ be
vertices incident with $v$ distinct from $v_1$ and $v_2$, and let $g$ be the face with the angle $v_3vv_4$.  The face $f$
sends $1/6$ to $g$ for each such angle with $v_3,v_4\not\in V(C)$ and $\deg(v_3)=\deg(v_4)=3$.
\item[(R2)] For each non-outer face $f$ of $G$ and for each incident edge $uv$ such that $u,v\not\in V(C)$, $\deg(u)=\deg(v)=4$, neither $u$ nor $v$ has
two neighbors of degree three not belonging to $V(C)$, and the other face $g$ incident with $uv$ is poor, the face $f$ sends $1/6$ to $g$.
\end{itemize}
Let $\ch_1$ denote the charge after performing the redistribution according to the rules (R0), (R1), and (R2).

Consider a $4$-face $f$ bounded by a cycle $W=v_1v_2v_3v_4$.  For integers $d_1$, \ldots, $d_4$, we say that $f$ is a $(d_1,d_2,d_3,d_4)$-face if
$v_i\not\in V(C)$ and $\deg(v_i)=d_i$ for $i=1,\ldots, 4$.  We say that $f$ is a $(\ge\!d_1,d_2,d_3,d_4)$-face if $v_2,v_3,v_4\not\in V(C)$,
$\deg(v_1)\ge d_1$ or $v_1\in V(C)$, $\deg(v_2)=d_2$, $\deg(v_3)=d_3$, and
$\deg(v_4)=d_4$; and similarly for other combinations.
We say that a vertex $v\in V(W)$ is \emph{rich} (with respect to the face $f$), if $\deg(v)\ge 5$ or $v\in V(C)$, and some neighbor $u$ of $v$ in $W$
satisfies $u\not\in V(C)$ and $\deg(u)\le 4$.
We say that a $(\ge\!5,3,4,3)$-face is \emph{very light}, and that very light faces as well as
$(\ge\!5,3,\ge\!5,3)$-faces, $(\ge\!5,4,\ge\!5,3)$-faces with the vertex of degree $4$ adjacent to two vertices of degree three not belonging to $V(C)$,
$(\ge\!5,4,4,3)$-faces, and $(\ge\!5,4,3,4)$-faces are \emph{light}.

\begin{lemma}\label{lemma-phase1}
Let $G$ with the outer face bounded by a cycle $C$ be as described at the beginning of Section~\ref{sec-discharging}.
If $G$ does not contain a $(1,4)$-reducible induced subgraph disjoint from $C$ with at most $19$ vertices,
then $\ch_1(v)\ge \max(\ch_0(v),0)$ for all $v\in V(G)$,
$\ch_1(f)\ge 0$ for each face $f$ of $G$ of length at least five, and
each $4$-face $f$ with $n_r$ rich incident vertices satisfies one of the following:
\begin{itemize}
\item $\ch_1(f)=-\frac{1}{2}n_r$ and $f$ is very light; or
\item $-\frac{1}{3}n_r\le \ch_1(f)<-\frac{1}{6}n_r$ and $f$ is light; or
\item $-\frac{1}{6}n_r\le \ch_1(f)<0$ and each rich vertex $v$ incident with $f$ has a $(v,C)$-good neighbor incident with $f$ and using no bud,
and this neighbor is $(v,C)$-excellent unless $f$ is a $(\ge\!5,\ge\!5,4,3)$-face, or
\item $\ch_1(f)\ge 0$.
\end{itemize}
\end{lemma}
\begin{proof}
Consider a vertex $v\in V(G)$.  If $v\not\in V(C)$, then since the minimum degree of $G_0$ is at least three,
we have $\deg(v)\ge 3$.  Note that $v$ sends no charge.  If $\deg(v)=3$, then
$v$ receives $3\times \frac{1}{3}$ from incident faces by (R0), and thus
$\ch_1(v)=\ch_0(v)+1=0$.  If $\deg(v)\ge 4$, then $\ch_1(v)=\ch_0(v)=\deg(v)-4\ge 0$.
If $v\in V(C)$, then $\deg(v)\ge 2$ since $C$ is a cycle.  If $\deg(v)=2$, then
$v$ receives $1/3$ by (R0) and $\ch_1(v)=\ch_0(v)+1/3=0$.  If $\deg(v)\ge 3$,
then $\ch_1(v)=\ch_0(v)=\deg(v)-7/3>0$.

Consider now a face $f$.  If $f$ is the outer face, then $\ch_1(f)=\ch_0(f)=0$.
Hence, we can assume that $f$ is not the outer face.
Let $A_0$ denote the set of angles of $f$ for that (R0) applies,
let $A'_0$ denote the angles in $A_0$ whose tips do not belong to $V(C)$,
let $A_1$ denote the set of angles of $f$ for that (R1) applies,
and let $E_2$ denote the set of edges of $f$ for that (R2) applies.
Note the following:
\begin{itemize}
\item $A_0\cap A_1=\emptyset$ and the edges of $E_2$ are not incident with the tips of angles of $A_0\cup A_1$.
\item The tips of angles of $A'_0$ form an independent set (by Lemma~\ref{lemma-redusmall}).
\item If $u$ is the tip of an angle $a_u\in A'_0$ and $v$ is the tip of an angle $a_v$ consecutive to $a_u$ on $f$,
then $a_v\not\in A_1$ (as otherwise $v$ would be incident with three $(v,C)$-good vertices with stalks satisfying (a),
and thus $G$ would contain a $(1,4)$-reducible induced subgraph disjoint from $C$ with at most $19$ vertices by Lemma~\ref{lemma-mainredu}).
\end{itemize}
Consequently, $|A_0|+|A_1|+|E_2|\le |f|$, $|A'_0|\le \lfloor |f|/2\rfloor$,
 and if $|A_0|+|A_1|+|E_2|=|f|$, then $A'_0=\emptyset$.

Let us first consider the case that $A_0\neq A'_0$, i.e., some vertex in $V(C)$ of degree two is incident with $f$.
Since $G\neq C$ and $G$ is connected, it follows that at least two angles of $f$ have a tip in $V(C)$ of degree at least three,
and consequently $|A_0|+|A_1|+|E_2|\le |f|-2$.  Since $G$ is triangle-free and all $(\le\!5)$-cycles in $G$ bound faces, if $|f|=4$, then
$f$ would be incident with a vertex not in $V(C)$ of degree two, contradicting the assumption that $G_0$ has minimum degree at least three.
Hence, $|f|\ge 5$, and
\begin{align*}
\ch_1(f)&=\ch_0(f)-|A_0|/3-|A_1|/6-|E_2|/6\\
&\ge |f|-4-(|f|-2)/3\\
&=\frac{2|f|-10}{3}\ge 0.
\end{align*}
Hence, we can assume that $A_0=A'_0$, and in particular $|A_0|\le \lfloor |f|/2\rfloor$ and either $|A_0|=0$ or $|A_0|+|A_1|+|E_2|\le |f|-1$.  Consequently,
\begin{align*}
\ch_1(f)&=\ch_0(f)-|A_0|/3-|A_1|/6-|E_2|/6\\
&\ge \ch_0(f)-(|f|-1)/6-\lfloor |f|/2\rfloor/6\\
&=\frac{5|f|-\lfloor |f|/2\rfloor-23}{6}.
\end{align*}
If $|f|\ge 5$, this implies $\ch_1(f)\ge 0$.

Hence, assume $|f|=4$.  Let $v_1v_2v_3v_4$ be the cycle bounding $f$.
Suppose first that $f$ is poor.  Note that $f$ is incident with at most one vertex of degree three by Lemma~\ref{lemma-mainredu},
since if say $\deg(v_1)=3$ and $\deg(v_i)=3$ for some $i\in \{2,3,4\}$, then $v_1$ has two $(v_1,C)$-good neighbors $v_2$ and $v_4$,
at most one of them using a bud.  Hence, we can assume that $\deg(v_2)=\deg(v_3)=\deg(v_4)=4$ and
either $\deg(v_1)=3$, or $\deg(v_1)=4$ and $v_1$ has two neighbors of degree three not belonging to $V(C)$.
If for some $i\in\{2,3,4\}$, the vertex $v_i$ had a neighbor $y\neq v_1$ of degree three not belonging to $V(C)$,
then either $v_i$ would have three $(v_i,C)$-good neighbors, at most one of them using a bud,
or (when $\deg(v_1)=4$ and $i\neq 3$) $v_1$ would have three $(v_1,C)$-good neighbors not using buds;
and then Lemma~\ref{lemma-mainredu} would contradict the assumption that $G$ does not contain a $(1,4)$-reducible induced subgraph disjoint from $C$ with at most $19$ vertices.
Hence, none of the vertices $v_2$, $v_3$, and $v_4$ has a neighbor of degree three not belonging to $V(C)$,
and thus $|A_0|+|A_1|=1$.  Similarly, Lemma~\ref{lemma-mainredu} implies that $|E_2|=0$.
Consequently, $f$ sends at most $1/3$ by (R0) and (R1), and receives $2\times \frac{1}{6}$ by (R2) over edges $v_2v_3$
and $v_3v_4$, ensuring that $\ch_1(f)\ge \ch_0(f)=0$.

Therefore, suppose that $f$ is not poor.  Let us now distinguish cases depending on $|A_0|$.
If $|A_0|=2$, then by Lemmas~\ref{lemma-redusmall} and \ref{lemma-mainredu}, $f$ is either a $(\ge\!5,3,4,3)$-face or a $(\ge\!5,3,\ge\!5,3)$-face.
In the former case, $f$ is very light, sends $2\times \frac{1}{3}$ by (R0) and receives $1/6$ by (R1),
and $\ch_1(f)=-\frac{1}{2}=-\frac{1}{2}n_r$.  In the latter case, $f$ is light and sends $2\times \frac{1}{3}$ by (R0),
and thus $\ch_1(f)=-\frac{2}{3}=-\frac{1}{3}n_r$.

If $|A_0|=1$, then since $f$ is not poor, we have the following possibilities:
\begin{itemize}
\item $f$ is a $(\ge\!5,\ge\!5,\ge\!5,3)$-face: $n_r=2$, and $\ch_1(f)=-\frac{1}{3}=-\frac{1}{6}n_r$, and $v_4$ is a $(v_i,C)$-excellent neighbor
of $v_i$ for $i\in\{1,3\}$.
\item $f$ is a $(\ge\!5,4,\ge\!5,3)$-face: If the vertex of degree $4$ is adjacent to two vertices of degree three not belonging to $V(C)$,
then $|A_1|=1$, $f$ is light, $n_r=2$, and $\ch_1(f)=-\frac{1}{3}-\frac{1}{6}>-\frac{1}{3}n_r$.  Otherwise
$\ch_1(f)=-\frac{1}{3}=-\frac{1}{6}n_r$ and $v_4$ is a $(v_i,C)$-excellent neighbor of $v_i$ for $i\in\{1,3\}$.
\item $f$ is a $(\ge\!5,\ge\!5,4,3)$-face: We have $A_1=\emptyset$ by Lemma~\ref{lemma-mainredu}, $n_r=2$, and $\ch_1(f)=-\frac{1}{3}=-\frac{1}{6}n_r$,
$v_3$ is a $(v_2,C)$-good neighbor of $v_2$, and $v_4$ is a $(v_1,C)$-good neighbor of $v_1$.
\item $f$ is a $(\ge\!5,4,4,3)$-face or a $(\ge\!5,4,3,4)$-face: We have $A_1=E_2=\emptyset$ by Lemma~\ref{lemma-mainredu},
$n_r=1$, $f$ is light, and $\ch_1(f)=-\frac{1}{3}=-\frac{1}{3}n_r$.
\end{itemize}

Finally, if $A_0=\emptyset$, we have the following possibilities.
\begin{itemize}
\item $f$ is a $(4,4,4,4)$-face: Since $f$ is not poor, we have $A_1=\emptyset$.  By Lemma~\ref{lemma-spec4}, we have $E_2=\emptyset$.
Consequently, $\ch_1(f)=0$.
\item $f$ is a $(\ge\!5,4,4,4)$-face: By Lemma~\ref{lemma-mainredu}, we have $|A_1|+|E_2|\le 1$, and thus $\ch_1(f)\ge -\frac{1}{6}=-\frac{1}{6}n_r$.
Furthermore, either $\ch_1(f)=0$, or $|A_1|+|E_2|=1$ and $v_2$ or $v_4$ is a $(v_1,C)$-excellent neighbor of the rich vertex $v_1$.
\item $f$ is a $(\ge\!5,\ge\!5,4,4)$-face: Note that $|A_1|+|E_2|\le 2$, and thus $\ch_1(f)\ge -2\cdot\frac{1}{6}=-\frac{1}{6}n_r$.
Furthermore, either $\ch_1(f)=0$, or $|A_1|+|E_2|\ge 1$ and
$v_3$ is a $(v_2,C)$-excellent neighbor of $v_2$, and $v_4$ is a $(v_1,C)$-excellent neighbor of $v_1$.
\item $f$ is a $(\ge\!5,4,\ge\!5,4)$-face: Then $\ch_1(f)\ge -2\cdot\frac{1}{6}=-\frac{1}{6}n_r$,
and either $\ch_1(f)=0$, or $A_1$ contains a $(v_i,C)$-excellent neighbor of $v_i$ for $i\in\{1,3\}$.
\item $f$ is a $(\ge\!5,\ge\!5,\ge\!5,4)$-face: Then $\ch_1(f)\ge -\frac{1}{6}>-\frac{1}{6}n_r$, and either $\ch_1(f)=0$, or $v_4\in A_1$ and $v_4$
is a $(v_i,C)$-excellent neighbor of $v_i$ for $i\in\{1,3\}$.
\item $f$ is a $(\ge\!5,\ge\!5,\ge\!5,\ge\!5)$-face: Then $\ch_1(f)=0$.
\end{itemize}

\end{proof}

Now, we do one more redistribution of the charge, according to the following rule:
\begin{itemize}
\item[(R3)] Let $f$ be a $4$-face with $n_r$ incident rich vertices.  If $v$ is a rich vertex incident with $f$, then
$v$ sends $-\ch_1(f)/n_r$ to $f$.
\end{itemize}
Let $\ch_2$ denote the charge after performing the redistribution according to all the rules (R0)---(R3).
Lemma~\ref{lemma-phase1} and the rule (R3) imply that all faces as well as vertices not in $V(C)$ of degree at most $4$ have non-negative charge.
We now need to argue about vertices of degree at least $5$.

\begin{lemma}\label{lemma-fiveopplight}
Let $G$ with the outer face bounded by a cycle $C$ be as described at the beginning of Section~\ref{sec-discharging}.
Let $v\not\in V(C)$ be a vertex of degree $5$, incident with angles of faces $f_1$, \ldots, $f_5$
in order.  If $G$ does not contain a $(1,4)$-reducible induced subgraph disjoint from $C$ with at most $25$ vertices,
then at most one of faces $f_1$ and $f_3$ is a light $4$-face.
\end{lemma}
\begin{proof}
Let $v_1vv_2$ and $v_3vv_4$ be the angles of $f_1$ and $f_3$ incident with $v$.  If both $f_1$ and $f_3$ were light, then
the inspection of the definition of a light face shows that $v_1$, \ldots, $v_4$ would be $(v,C)$-good vertices,
no two of them using the same bud, and thus $G$ would contain a $(1,4)$-reducible induced subgraph
disjoint from $C$ with at most $25$ vertices by Lemma~\ref{lemma-mainredu}, in contradiction to the assumptions.
\end{proof}

\begin{lemma}\label{lemma-fiveoppgood}
Let $G$ with the outer face bounded by a cycle $C$ be as described at the beginning of Section~\ref{sec-discharging}.
Let $v\not\in V(C)$ be a vertex of degree $5$, adjacent to vertices $v_1$, \ldots, $v_5$ in order.
Let $f_1$, $f_3$, and $f_4$ be the faces with angles $v_1vv_2$, $v_3vv_4$, and $v_4vv_5$, respectively.
If $G$ does not contain a $(1,4)$-reducible induced subgraph disjoint from $C$ with at most $25$ vertices, $f_1$ is a light $4$-face, and $v$ sends a positive amount of charge
to both $f_3$ and $f_4$ by the rule (R3), then $v_4$ is $(v,C)$-excellent.
\end{lemma}
\begin{proof}
By Lemma~\ref{lemma-fiveopplight}, neither $f_3$ nor $f_4$ is light.  Since $v$ sends a positive amount of charge to both $f_3$ and $f_4$ by the rule (R3),
Lemma~\ref{lemma-phase1} implies that both $f_3$ and $f_4$ are incident with a $(v,C)$-good vertex using no bud.
Since $f_1$ is light, both $v_1$ and $v_2$ are $(v,C)$-good vertices, at most one of which uses a bud.
By Lemma~\ref{lemma-mainredu}, we conclude that at most one of $v_3$ and $v_5$ is a $(v,C)$-good vertex using no bud.
Hence, $v_4$ is a $(v,C)$-good vertex using no bud, and neither $v_3$ nor $v_5$ is a $(v,C)$-good vertex using no bud.
Also, by Lemma~\ref{lemma-phase1}, $v_4$ is $(v,C)$-excellent
unless both $f_3$ and $f_4$ are $(\ge\!5,\ge\!5,4,3)$-faces.
But then $v_4$ is also $(v,C)$-excellent, since it is adjacent to two vertices of degree three not belonging to $C$.
\end{proof}

\begin{corollary}\label{cor-charge5}
Let $G$ with the outer face bounded by a cycle $C$ be as described at the beginning of Section~\ref{sec-discharging}.
Let $v\not\in V(C)$ be a vertex of degree $5$.  If $G$ does not contain a $(1,4)$-reducible induced subgraph disjoint from $C$
with at most $25$ vertices, then $\ch_2(v)\ge 0$.
\end{corollary}
\begin{proof}
Let $v_1$, \ldots, $v_5$ be vertices adjacent to $v$ in order, and let $f_1$, \ldots, $f_5$ be faces incident with
angles $v_1vv_2$, \ldots, $v_5vv_1$.  If none of $f_1$, \ldots, $f_5$ is light, then by Lemma~\ref{lemma-phase1},
$v$ sends at most $1/6$ to each incident face by (R3), and $\ch_2(v)\ge \ch_0(v)-5\times\frac{1}{6}=1/6$.
Hence, we can assume that $f_1$ is light.  By Lemma~\ref{lemma-fiveopplight}, neither $f_3$ nor $f_4$ is light,
and at most one of $f_2$ and $f_5$ is light; by symmetry, we can assume that $f_5$ is not light.

By Lemmas~\ref{lemma-5redu} and \ref{lemma-fiveoppgood}, if $f_1$ is very light, then $v$ does not send charge to both $f_3$ and $f_4$
due to (R3).  If $f_1$ is not very light, then $v$ sends at most $1/3$ to $f_1$ due to (R3).  In either case,
$v$ sends at most $\max(1/2+1/6,1/3+2\times \frac{1}{6})=2/3$ to $f_1$, $f_3$, and $f_4$
in total.  If $f_2$ is not light, then $v$ sends at most $1/6$ to each of $f_2$ and $f_5$, and
thus $\ch_2(v)\ge \ch_0(v)-2/3-2\times\frac{1}{6}=0$.

Hence, we can assume that $f_2$ is light.  By Lemma~\ref{lemma-5redu}, $f_1$ and $f_2$ are not both very light.
Let us first consider the case that neither $f_1$ nor $f_2$ is very light,
and thus $v$ sends at most $2\times \frac{1}{3}$ to $f_1$ and $f_2$ in total by (R3).
If $v$ does not send charge to all of $f_3$, $f_4$, and $f_5$, then
$\ch_2(v)\ge \ch_0(v)-2/3-2\times\frac{1}{6}=0$.  If $f$ sends charge to all of $f_3$, $f_4$, and $f_5$,
then by Lemma~\ref{lemma-fiveoppgood} the vertices $v_4$ and $v_5$ are $(v,C)$-excellent, and thus they are $(v,C)$-good using
no bud.  Since $f_1$ is light, both $v_1$ and $v_2$ are $(v,C)$-good (using at most one bud).
By Lemma~\ref{lemma-mainredu}, we conclude that $G$ contains a $(1,4)$-reducible induced subgraph disjoint from $C$
with at most $25$ vertices, which is a contradiction.

Hence, we can by symmetry assume that $f_1$ is very light (and $f_2$ is not),
and thus $f$ sends at most $1/2+1/3=5/6$ to $f_1$ and $f_2$ in total by (R3).
Since $v_3$ is not $(v,C)$-excellent by Lemma~\ref{lemma-5redu}, the inspection of the definition of a light face shows that
$f_2$ is a $(5,4,4,3)$-face with $\deg(v_3)=4$.  Since $v_3$ is not $(v,C)$-excellent, it follows that $f_3$ is not
a $(\ge\!5,\ge\!5,4,3)$-face.  Since neither $v_3$ nor $v_4$ is $(v,C)$-excellent, Lemma~\ref{lemma-phase1} implies that $v$
does not send any charge to $f_3$ by (R3).   Since $v_1$, $v_2$, and $v_3$ are $(v,C)$-good and only $v_3$ uses a bud $v_2$,
Lemma~\ref{lemma-mainredu} implies that neither $v_4$ nor $v_5$ is $(v,C)$-good using no bud.  By Lemma~\ref{lemma-phase1}, $v$ does not send
any charge to $f_4$.  Consequently, $\ch_2(v)\ge \ch_0(v)-5/6-1/6=0$.
\end{proof}

\begin{lemma}\label{lemma-nonneg}
Let $G$ with the outer face bounded by a cycle $C$ be as described at the beginning of Section~\ref{sec-discharging}.
If $G$ does not contain a $(1,4)$-reducible induced subgraph disjoint from $C$ with at most $31$ vertices, then
$\ch_2(v)\ge 0$ for all $v\in V(G)$ and $\ch_2(f)\ge 0$ for all $f\in F(G)$.
\end{lemma}
\begin{proof}
We have $\ch_2(f)\ge 0$ by Lemma~\ref{lemma-phase1} if $f$ is the outer face or $|f|\ge 5$,
and by (R3) if $|f|=4$.  We also have $\ch_2(v)\ge 0$ if $v\not\in V(C)$ and $\deg(v)\le 4$
or if $v\in V(C)$ and $\deg(v)=2$ by Lemma~\ref{lemma-phase1}, and
if $v\not\in V(C)$ and $\deg(v)=5$ by Corollary~\ref{cor-charge5}.

If $v\in V(C)$ and $\deg(v)\ge 3$, then note that at most $\deg(v)-3$ faces incident with $v$ are very light (since $v$ has two neighbors
belonging to $C$), and thus
$\ch_2(v)\ge \ch_1(v)-(\deg(v)-3)/2-2\cdot \frac{1}{3}=(\deg(v)-7/3)-(\deg(v)-3)/2-2/3=(\deg(v)-3)/2\ge 0$
by Lemma~\ref{lemma-phase1} and the rule (R3).

Finally, let us consider the case that $v\not\in V(C)$ and $\deg(v)\ge 6$.  If $\deg(v)\ge 8$, then
$\ch_2(v)\ge \ch_1(v)-\deg(v)/2=(\deg(v)-4)-\deg(v)/2\ge 0$ by Lemma~\ref{lemma-phase1} and the rule (R3).
If $\deg(v)=7$, then note that $v$ has at most $5$ neighbors of degree three not belonging to $V(C)$
by Lemma~\ref{lemma-mainredu}, and thus $v$ is incident with at most $4$ very light faces.
Consequently,
$\ch_2(v)\ge \ch_0(v)-4\cdot \frac{1}{2}-3\cdot \frac{1}{3}=0$
by Lemma~\ref{lemma-phase1} and the rule (R3).

Suppose now that $\deg(v)=6$, let $v_1$, \ldots, $v_6$ be neighbors of $v$ in order,
and let $f_i$ denote the face incident with the angle $v_ivv_{i+1}$ for $i=1,\ldots, 6$ (where $v_7=v_1$).
If no incident face is very light, then $\ch_2(v)\ge \ch_0(v)-6\cdot \frac{1}{3}=0$
by Lemma~\ref{lemma-phase1} and the rule (R3).
Hence, we can assume that $f_1$ is very light, and thus $v_1$ and $v_2$ are vertices of degree three not belonging to $V(C)$.

If the face $f_3$ is light, then both $v_3$ and $v_4$ are $(v,C)$-good using at most one bud, incident with $f_3$.
By Lemma~\ref{lemma-mainredu}, we conclude that $v_5$ and $v_6$ are not $(v,C)$-good using only buds not incident with $f_3$,
and consequently $f_5$ and $f_6$ are not light and $f_4$ is not very light.  Furthermore, if $f_3$ is very light, then
$v_5$ is not $(v,C)$-good at all, and thus $f_4$ is not light.  Hence, at most $\max(2\times \frac{1}{3},1/2+1/6)=2/3$
is sent to $f_3$ and $f_4$ in total by (R3), and
$\ch_2(v)\ge \ch_0(v)-2\cdot \frac{1}{2}-2/3-2\cdot \frac{1}{6}=0$.

Hence, we can assume that $f_3$ is not light, and by symmetry $f_5$ is not light.  Analogously, if one of the faces
$f_2$, $f_4$, and $f_6$ is very light, we can assume that the other two are not light.  Hence, we send at most $1/2+2\times \frac{1}{6}<1$ to $f_1$, $f_3$, and $f_5$ in total,
and at most $\max (1/2+2\times \frac{1}{6}, 3\times \frac{1}{3})=1$ to $f_2$, $f_4$, and $f_6$ in total, and thus
$\ch_2(v)>\ch_0(v)-1-1=0$.
\end{proof}

\begin{proof}[Proof of Theorem~\ref{thm-main}]
Let $G_0$ be a plane triangle-free graph.  We apply Lemma~\ref{lemma-redu} (with $g=k=4$ and $b=31$) to show that $G_0$ is
weighted $\eps$-flexible (for fixed $\eps>0$ corresponding to the given values of $g$, $k$, and $b$)
with any assignment of lists of size $4$.  Since every subgraph of $G_0$ is planar and triangle-free,
it suffices to prove that $G_0$ contains a $(1,4)$-reducible induced subgraph with at most $31$ vertices.

Without loss of generality, we can assume that $G_0$ is connected, and by Lemma~\ref{lemma-redusmall}, we can assume
that all vertices of $G_0$ have degree at least three. 
Let $G$ with the outer face bounded by a cycle $C$ and the assignment $\ch_0$ of charges to its vertices and faces
be as described at the beginning of Section~\ref{sec-discharging}.
Recall that the sum of these charges is negative.  Redistributing the charge according to the rules (R0)---(R3)
gives us the charge assignment $\ch_2$ with the same (negative) sum of charges, and thus some vertex or face has negative charge.
Lemma~\ref{lemma-nonneg} implies that $G$ contains a $(1,4)$-reducible induced subgraph $H$ disjoint from $C$ with at most $31$ vertices.
As we observed before, $H$ is also a $(1,4)$-reducible induced subgraph of $G_0$.
\end{proof}

\bibliographystyle{siam}
\bibliography{req-trfree}

\end{document}